\documentclass[12pt]{amsart}
\usepackage{amssymb,amsmath,amscd,amsthm,pb-diagram}
\title{Some remarks on Nil groups in algebraic K-theory}
\author{James F. Davis}
\thanks{Partially supported by a grant from the National Science Foundation}
\address{Department of Mathematics \\
Indiana University \\
Bloomington, IN 47405 \\ USA} \email{jfdavis@indiana.edu}
\date{}

\newcommand{\cB}{{\mathcal{B}}}
\newcommand{\cC}{{\mathcal{C}}}

\newcommand{\cF}{{\mathcal{F}}}
\newcommand{\cG}{{\mathcal{G}}}
\newcommand{\cH}{{\mathcal{H}}}

\newcommand{\cM}{{\mathcal{M}}}
\newcommand{\cN}{{\mathcal{N}}}

\newcommand{\bfE}{{\mathbf E}}
\newcommand{\bfK}{{\mathbf K}}

\newcommand{\bfN}{{\mathbf N}}

\newcommand{\bfX}{{\mathbf X}}
\newcommand{\bfY}{{\mathbf Y}}
\newcommand{\bfZ}{{\mathbf Z}}
\newcommand{\bfnil}{{\mathbf {Nil}}}
\newcommand{\da}{{\downarrow}}

\newcommand{\Q}{{\mathbb Q}}
\newcommand{\R}{{\mathbb R}}
\newcommand{\Z}{{\mathbb Z}}

\newcommand{\g}{{\Gamma}}
\newcommand{\go}{{\Gamma_0}}

\DeclareMathOperator{\hocolim}{{hocolim}}

\DeclareMathOperator{\id}{Id}
\DeclareMathOperator{\map}{map}
\DeclareMathOperator{\mor}{mor}
\DeclareMathOperator{\nil}{Nil}
\DeclareMathOperator{\op}{op}
\DeclareMathOperator{\Or}{Or}
\DeclareMathOperator{\pt}{pt}

\DeclareMathOperator{\fin}{fin}
\DeclareMathOperator{\cfin}{cfin}
\DeclareMathOperator{\csub}{CSub}
\DeclareMathOperator{\vc}{vc}
\DeclareMathOperator{\all}{all}
\DeclareMathOperator{\ab}{Abelian~Groups}
\DeclareMathOperator{\rings}{Rings}

\DeclareMathOperator{\spaces}{Top}

\DeclareMathOperator{\spectra}{Spectra}
\DeclareMathOperator{\initial}{initial}

\newtheorem{theorem}{Theorem}
\newtheorem*{theorem1}{Theorem 1}
\newtheorem*{theorem5}{Theorem 5}

\newtheorem{corollary}[theorem]{Corollary}
\newtheorem{conjecture}[theorem]{Conjecture}
\newtheorem{lemma}[theorem]{Lemma}
\theoremstyle{definition}

\newtheorem{remark}[theorem]{Remark}

\numberwithin{equation}{theorem}
\begin{document}
\begin{abstract}
This note explains consequences of recent work of Frank Quinn for computations of Nil groups in algebraic K-theory, in particular the Nil groups occurring in the K-theory of polynomial rings, Laurent polynomial rings, and the group ring of the infinite dihedral group.
\end{abstract}
\maketitle

\section{Statement of Results}
Let $R$ be a ring with unit.  For an integer $q$, let $K_qR$ be the algebraic $K$-group of Bass and Quillen.  Bass defines the NK-groups 
$$
NK_q(R) = \ker(K_qR[t] \to K_qR)
$$
where the map on $K$-groups is induced by the ring map 
$$R[t] \to R, \quad f(t) \mapsto f(0).
$$
The $NK$-groups are often called Nil-groups because they are related to nilpotent endomorphisms of projective $R$-modules.

Let $G$ be a group.  Let $\Or G$ be its the orbit category; objects are $G$-sets $G/H$ where $H$ is a subgroup of $G$ and morphisms are $G$-maps.   Davis-L\"uck \cite{DL98} define a functor $\bfK : \Or G \to \spectra$  with the key property  $\pi_q \bfK(G/H) = K_q(RH)$.  The utility of such a functor is to allow the definition of an equivariant homology theory, indeed for a $G$-CW-complex $X$, one defines
$$
H^G_q(X; \bfK) = \pi_q(\map_G(-,X)_+ \wedge_{\Or G} \bfK(-))
$$
(see \cite[section 4 and 7]{DL98} for basic properties).  Note that $\map_G(G/H,X) = X^H$ is the fixed point functor and that the ``coefficients'' of the homology theory are given by $H^G_q(G/H; \bfK) = K_q(RH)$.

A {\em family $\cF$ of subgroups of $G$} is a nonempty set of subgroups closed under subgroups and conjugation.  For such a family, $E_{\cF}$ (short for $E_{\cF}G$)  is the classifying space for $G$-actions with isotopy in $\cF$.  It is characterized up to $G$-homotopy type as a $G$-CW-complex so that $E_{\cF}^H$ is contractible for subgroups $H \in \cF$ and is empty for subgroups $H \not \in \cF$.

Consider the following families of subgroups of $G$:  
$$
1 \subset \fin \subset \cfin \subset \vc \subset \all
$$
Here 
\begin{align*}
\fin & = \{\text{finite subgroups} \} \\
\cfin & = \{\text{cyclic subgroups} \} \cup \{\text{finite subgroups} \}  \\
\vc & = \{\text{virtually cyclic subgroups} \}
\end{align*}
Note that a model for $E_{\all}G$ is $ G/G = \pt$, so $H^G_q(E_{\all};\bfK) = K_q(RG)$.  Note also $E_1G = EG$ by definition.

The Farrell-Jones isomorphism conjecture in $K$-theory \cite{FJ93}, as reinterpreted in \cite{DL98}, states that $H^G_*(E_{\all},E_{\vc};\bfK) = 0$.  (Here and elsewhere, given a map $f: A \to B$, we write $H_*(B,A)$ as a shorthand notation for $H_*(M(f),A)$ where $M(f)$ is the mapping cylinder.)

\begin{theorem}  \label{Zn}
 Let $\cM$ be the set of maximal cyclic subgroups of $\Z^n$.
$$
H^{\Z^n}_q(E_{\vc}, E_{1};\bfK)  
 \cong  \bigoplus_{C \in \cM}  ~ \bigoplus_{n-1 \geq i \geq 0} 2\binom{n-1}{i} NK_{q-i}R.
$$
\end{theorem}


See Remark \ref{precise} for a description of the isomorphism in Theorem \ref{Zn}.
Assuming a theorem of Frank Quinn \cite{Q1}, proven using controlled topology, we will show the following corollaries in Section \ref{Quinn}.

\begin{corollary}  \label{laurent}
\begin{multline*}
K_qR[t_1,t_1^{-1}, \dots, t_n, t_n^{-1}]  = K_qR[\Z^n]\\
 \cong \bigoplus_{n \geq i \geq 0} \binom{n}{i} K_{q-i}R~ \oplus  ~\bigoplus_{C \in \cM}  ~ \bigoplus_{n-1 \geq i \geq 0} 2\binom{n-1}{i} NK_{q-i}R.
\end{multline*}
\end{corollary}

The isomorphism is described explicitly in Remark \ref{explicit}.

\begin{corollary} \label{kregular} Let $q$ be an integer.
If $NK_jR = 0$ when $q \geq j \geq q-n+1$, then $$K_qR[t_1, \dots, t_n] = K_qR.$$
\end{corollary}

These corollaries was proved in \cite{CHW} in the case where $R$ is a commutative ring containing the rationals.  Their techniques are from algebraic geometry and commutativity is crucial for their proof.

This corollary is a partial answer to Bass' question \cite[Question $(IV)_n$]{BQ} asking does $NK_qR = 0$ imply that $K_qR[t_1,t_2] = K_qR$?  However, recently  Corti\~nas, Haesemeyer, and Weibel \cite{CHW} recently showed the answer is no in general.

Corollary \ref{countable} in Section \ref{Quinn} is a special case of the following conjecture. 

\begin{conjecture}
Let $\cM_+$ be the set of maximal cyclic subgroups of $\Z^n$ with a generator having  all positive coordinates.  Then
$$
N^nK_qR \cong \bigoplus_{C \in \cM_+} \bigoplus_{n-1 \geq i \geq 0} \binom{n-1}{i} NK_{q-i}R.
$$
\end{conjecture}

This conjecture was proven by Corti\~nas, Haesemeyer, and Weibel \cite[Corollary 4.2]
{CHW} in the case where $R$ is a commutative ring containing the rationals.

We need some notation to state our next result.  Given a homomorphism    $\varphi : G \to  \Gamma$ and a family $\cF$ of subgroups of $\Gamma$, let $\varphi^*\cF$ be the smallest family of subgroups of $G$ containing $\varphi^{-1}H$ for each $H \in \cF$.  If $\alpha : A \to A$ is a ring automorphism, let $A_{\alpha}[t]$ be the twisted polynomial ring and define 
$$
NK_q(A,\alpha) = \ker(K_q(A_{\alpha}[t]) \to K_q(A)).
$$
These groups were first defined by Farrell.  

Let $D_\infty = \Z_2 * \Z_2 = \langle a,b~|~a^2 = 1 = b^2 \rangle$ be the infinite dihedral group.

\begin{theorem} \label{infinite_dihedral}
Let $\varphi : \g \to D_{\infty}$ be an epimorphism of groups.  Let $F = \ker \varphi$. Choose $\hat t \in \g$ so that $\varphi(\hat t) = ab$.    Let $\alpha$ be the automorphism of $RF$ given by conjugation by $\hat t$.  Then
$$
H^\g_q(E_{\varphi^*\cfin},E_{\varphi^*\fin}; \bfK) \cong NK_q(RF,\alpha).
$$
\end{theorem}

To say that a group admits an epimorphism to the infinite dihedral group is equivalent to saying that it admits an amalgamated product decomposition $G_0 *_F G_1$ where $F$ is of index 2 in $G_0$ and $G_1$.

Theorem \ref{infinite_dihedral} is applied in \cite{DKR} to reduce the Farrell-Jones Conjecture in $K$-theory from the family of virtually cyclic groups to the family of finite-by-cyclic groups.

Assuming a theorem of Frank Quinn \cite{Q2}, we have the following corollary.  The notation in its statement  will be discussed in the next section.

\begin{corollary} \label{dv}
Let $\varphi : \g \to D_{\infty}$ be an epimorphism of groups and $\g = G_0 *_F G_1$ the 
corresponding amalgamated free product decomposition with $F = \ker \varphi$.  Choose $\hat t \in \g$ so that $\varphi(\hat t) = ab$, and let $\alpha : RF \to RF$ the automorphism given by conjugation by $\hat t$.  Then the Waldhausen Nil group is isomorphic to the Farrell Nil group:
$$
NK_q(RF; \widehat{R G_0}, \widehat{RG_1 }) \cong NK_q(RF, \alpha).
$$
\end{corollary}

A different, purely algebraic proof of this corollary is given in \cite{DKR}.

\section{NK-groups and relative homology}

To prove our main theorems we need a starting point.  This will be to recast the theorems of Bass-Heller-Swan, Bass, Quillen, Farrell-Hsiang, and ultimately Waldhausen \cite[Theorems 1 and 2]{W} in terms of $H^G_q(E_{\all},E_{\cF};\bfK)$ for suitable families $\cF$.

Let $\alpha : A \to A$ be a ring automorphism.  Waldhausen shows that the maps
\begin{align*}
i_+ : K_qA_{\alpha}[t] \to K_qA_{\alpha}[t,t^{-1}] &\text{ induced by } t \mapsto t  \\
i_- : K_qA_{\alpha^{-1}}[t] \to K_qA_{\alpha}[t,t^{-1}] &\text{ induced by } t \mapsto t^{-1}
\end{align*}
are monomorphisms.  Let
\begin{align*}
N_+K_q(A,\alpha) &= i_+(NK_q(A,\alpha)) \\
N_-K_q(A,\alpha^{-1}) &= i_-(NK_q(A,\alpha^{-1})) 
\end{align*}
Waldhausen shows there is a split injection
$$
i_+ \oplus i_-: N_+K_q(A,\alpha) \oplus N_-K_q(A,\alpha^{-1}) \to  K_q(A_{\alpha}[t,t^{-1}])
$$
and a Wang type exact sequence 
$$
\dots \to  K_qA \xrightarrow{1 - \alpha} K_qA \to \frac{K_qA_{\alpha}[t,t^{-1}]}{N_+K_q(A,\alpha) \oplus N_-K_q(A,\alpha^{-1}) }          \to \cdots
$$

Rephrasing this in terms of group theory, let $\go$ be a group which maps epimorphically to $\Z$.  Then $\go = F \rtimes_{\alpha} \Z$ and there is
a Wang type exact sequence
$$
\dots \to  K_q(RF) \xrightarrow{1 - \alpha} K_q(RF) \to \frac{K_q(R\go)}{N_+K_q(RF,\alpha) \oplus N_-K_q(RF,\alpha^{-1}) }          \to \cdots
$$

In this section we also consider in parallel the case $\g= G_0 *_F G_1$ where $F$ is a subgroup of $G_0$ and $G_1$, not necessarily of index 2.
  Let $\widehat{R G_i}$ be the $RF$-bimodule $R[G_i - F]$ and define
$$
NK_q(RF;\widehat{RG_0},\widehat{RG_1}) = \widetilde{\nil}_{q-1}(RF;\widehat{RG_0},\widehat{RG_1}).
$$  
This reduced Nil group is a subgroup of the $K$-theory of the exact category defined by Waldhausen \cite{W}
$$
K_q\nil(RF;\widehat{RG_0},\widehat{RG_1}) = \widetilde{\nil}_q(RF;\widehat{RG_0},\widehat{RG_1}) \oplus K_q(RF) \oplus K_q(RF)
$$
 When $q < 1$, we use the nonconnective version due to Bartels-L\"uck \cite[Section 10]{BL} to define the $K$-theory of the Nil category.  Waldhausen gave a split injection 
$$
NK_q(RF;\widehat{RG_0},\widehat{RG_1}) \to K_q(R\g) 
$$
and a Mayer-Vietoris type long exact sequence
$$
\dots \to  K_q(RF) \to K_q(RG_0) \oplus K_q(RG_1) \to K_q(R\g)/NK_q     \to \cdots
$$

The main purpose of this section is the statement and indication of the proof of the following lemma.

\begin{lemma} \label{dlnil} Let $\cF_0$ be the smallest family of subgroups of $\go$ containing $F$.  Let $\cF$ be the smallest family of subgroups of $\g$ containing $G_0$ and $G_1$.
\begin{enumerate}

\item  The following exact sequences are split, and hence short exact
\begin{align*}
H^{\go}_q(E_{\cF_0};\bfK) \to &H^{\go}_q(E_{\all};\bfK) \to H^{\go}_q(E_{\all},E_{\cF_0};\bfK) \\
H^{\g}_q(E_{\cF};\bfK) \to &H^{\g}_q(E_{\all};\bfK) \to H^{\g}_q(E_{\all},E_{\cF};\bfK)
\end{align*}

\item  These relative terms can be expressed in terms of Nil groups:
\begin{align*}
H^{\go}_q(E_{\all},E_{\cF_0};\bfK) & \cong NK_q(RF,\alpha) \oplus  NK_q(RF,\alpha^{-1}) \\
H^{\g}_q(E_{\all},E_{\cF};\bfK) & \cong NK_q(RF;\widehat{RG_0},\widehat{RG_1}) 
\end{align*}

\end{enumerate}
\end{lemma}

\begin{proof}
Define $E_{\cF_0} = E_{\cF_0}\go$ as a pushout of $\go$-spaces
\begin{equation} \label{squareone}
\begin{diagram}
\node{S^0 \times \go/F} \arrow{e} \arrow{s} \node{\go/F} \arrow{s}\\
\node{D^1 \times \go/F} \arrow{e} \node{E_{\cF_0}}
\end{diagram}
\end{equation}
where the upper horizontal ``attaching'' map takes $-1 \times F \mapsto F$ and $1 \times F\mapsto \hat t F$, where $\hat t \in \Z \subset \go$ is a generator.  Then $E_{\cF_0} = \R$, $E_{\cF_0}/\go = S^1$, and there is a Wang type long exact sequence
$$
\dots \to  K_q(RF) \xrightarrow{1 - \alpha} K_q(RF) \to H^{\go}_q(E_{\cF_0};\bfK)         \to \cdots
$$
Hence all we really need to do is to identify the map $H_q^{\go}(E_{\cF_0};\bfK) \to H_q^{\go}(E_{\all};\bfK) = K_q(R\go)$ with the split injection implicit in Waldhausen's work.  

For a $G$-CW-complex $X$, let 
$$
\bfK_{\%}(X) = \map_G(-,X)_+ \wedge_{\Or G} \bfK(-).
$$
This is the spectrum whose homotopy groups are $H_q^G(X;\bfK)$.  Since $\mor_{\Or \go}(\go/H,\go/F) = \map_{\go}(\go/H,\go/F)$ by definition,
Yoneda's lemma allows us to identify $\bfK_{\%}$ (with $G = \go$) applied to the square \eqref{squareone} with the commutative diagram of spectra
\begin{equation} \label{squaretwo}
\begin{diagram}
\node{\bfK(\go/F) \vee \bfK(\go/F)} \arrow{e} \arrow{s} \node{\bfK(\go/F)} \arrow{s} \\
\node{I_+ \wedge \bfK(\go/F)} \arrow{e} \node{\bfK_{\%}(E_{\cF_0})}
\end{diagram}
\end{equation}
This is a pushout diagram by \cite[Lemma 6.1]{DL98}.

Waldhausen gives a homotopy cartesian square 
\begin{equation} \label{squarethree}
\begin{diagram}
\node{\bfnil(RF)} \arrow{e} \arrow{s} \node{\bfK(RF)} \arrow{s} \\
\node{\bfK(RF)} \arrow{e} \node{\bfK(R\go)}
\end{diagram}
\end{equation}
(see Bartels-L\"uck \cite[Theorem 10.6]{BL} for the nonconnective version) and a split injection (up to homotopy)
$$
\bfK (RF)\vee \bfK (RF) \to \bfnil(R\go)
$$
with the homotopy groups of the cofiber being\linebreak$NK_{*+1}(RF,\alpha) \oplus NK_{*+1}(RF,\alpha^{-1})$.  Furthermore, Waldhausen shows that the composite of boundary map in the homotopy exact sequence of the square \eqref{squarethree} with the projection on the $NK$-groups 
$$
K_q(R\go) \to NK_q(RF,\alpha) \oplus NK_q(RF,\alpha^{-1})
$$
is a split surjection.

The square \eqref{squaretwo} maps to the square \eqref{squarethree}\footnote{There is a subtle point here. Waldhausen (see also \cite[Theorem 10.6]{BL}) shows that (\ref{squarethree}) is homotopy cartesian with respect to a certain natural transformation between the two functors corresponding to the two different ways from going to the upper left to the lower right.  However, this natural transformation involves the nilpotent structure and is the identity on the image of (\ref{squaretwo}).}.   Tracing through the above (examine the lower right corners!) gives a split short exact sequence 
$$
0 \to H_q^{\go}(E_{\cF_0};\bfK) \to H_q^{\go}(E_{\all};\bfK) \to NK_q(RF,\alpha) \oplus NK_q(RF,\alpha^{-1}) \to 0
$$
This gives a proof of the $\go$-part of the lemma.

  The proof of the $\g$-part of the lemma is quite similar.  Here we will only note that $E_{\cF}$ is constructed as a pushout
$$
\begin{diagram}
\node{S^0 \times \g/F} \arrow{e} \arrow{s} \node{\g/G_0 \amalg \g/G_1} \arrow{s}\\
\node{D^1 \times \g/F} \arrow{e} \node{E_{\cF}}
\end{diagram}
$$
and that $E_{\cF} = \R$ and $E_{\cF}/\g = [0,1/2]$.
\end{proof}

\section{Proof of Theorem \ref{Zn}}

Recall the statement of Theorem \ref{Zn}.

\begin{theorem1} 
 Let $\cM$ be the set of maximal cyclic subgroups of $\Z^n$.
$$
H^{\Z^n}_q(E_{\vc}, E_{1};\bfK)  
 \cong  \bigoplus_{C \in \cM}  ~ \bigoplus_{n-1 \geq i \geq 0} 2\binom{n-1}{i} NK_{q-i}R.
$$
\end{theorem1}

\begin{proof}
We make use of particular models for $E_1\Z^n = E_{\fin}\Z^n$ and $E_{\vc}\Z^n$, ensuring that they are $\Z^n$-CW-complexes.   Enumerate
 the maximal cyclic subgroups of $\Z^n$ as $\cM = \{C_0, C_1, C_2, \dots \}$.  Let 
$$
E_{1} = \R^n \times [0,\infty) \times I
$$ and define $E_{\vc}$ as the pushout
$$\begin{diagram}
\node{\coprod_{j=0}^\infty \R^n} \arrow{e} \arrow{s} \node{\coprod_{j=0}^\infty \R^n/(C_j \otimes \R)}\arrow{s}\\
\node{E_{1}} \arrow{e} \node{E_{\vc}}
\end{diagram}$$
where the $j$-th copy of $\R^n$ is identified with $\R^n \times \{j\} \times \{1\} \subset E_{1}$ and the $\Z^n$-actions on all spaces are induced by the translation action of $\Z^n$ on $\R^n$ and the trivial action on $[0,\infty) \times I$.

Let $\Z^{n-1}_j \subset \Z^n$ denote a subgroup so that $\Z^n = \Z^{n-1}_j \oplus C_j$.  Let   $\R_j = C_j \otimes \R \subset \R^n$ and $\R^{n-1}_j = \Z^{n-1}_j \otimes \R \subset \R^n$.  Then

\begin{align*}
H^{\Z^n}_q(E_{\vc},E_{1}; \bfK) & \xleftarrow{\cong} \bigoplus_j H^{\Z^n}_q(\R^n/\R_j,\R^n;\bfK)\\
& \cong \bigoplus_j \bigoplus_i H_i(\R^{n-1}_j/\Z^{n-1}_j) \otimes H^{C_j}_{q-i}(*,\R_j;\bfK) \\
&\cong  \bigoplus_{j} \bigoplus_{n-1 \geq i \geq 0} \binom{n-1}{i} H^{\Z}_{q-i}(*,\R ;\bfK)\\
& \cong \bigoplus_{j} \bigoplus_{n-1 \geq i \geq 0} \binom{n-1}{i} (NK_{q-i}R \oplus NK_{q-i}R),
\end{align*}
where the first isomorphism follows from the excision and disjoint union axioms, the second follows from the Atiyah-Hirzebruch Spectral Sequence \cite[Theorem 4.7]{DL98}
which collapses at $E_2$, the third from the homology of the torus, and the last from Lemma \ref{dlnil} with $\go = \Z$.
\end{proof}

An alternative method of computing $H^{\Z^n}_q(\R^n/\R_j,\R^n;\bfK)$ is to use the method of proof of Theorem \ref{infinite_dihedral} below.

\begin{remark} \label{precise}
The purpose of this remark is to make the map underlying the isomorphism of Theorem \ref{Zn} as explicit as possible.

For a maximal cyclic subgroup $C$ of $\Z^n$, let $ \csub$ be the family of subgroups of $\Z^n$ consisting of the subgroups of $C$.  Note that $E_{\csub}\Z^n = E(\Z^n/C)$ as $\Z^n$ spaces.

For an infinite cyclic group $C$ with generator $t$, there is a map 
$$
2NK_j(R) = N_jKR \oplus  N_jKR \to K_jRC
$$
induced by the two ring maps $R[t] \to RC$ given by $t \mapsto t$ and $t \mapsto t^{-1}$.

The $C$-component of the  isomorphism in Theorem \ref{Zn} is the composite of the maps
\begin{align*}
\bigoplus_i\binom{n-1}{i} 2NK_{q-i}R & \to \bigoplus_i\binom{n-1}{i} K_{q-i}RC \\
& \cong H_q(B(\Z^n/C); \bfK(RC))\\
& = H^{\Z^n}_q(E_{\csub};\bfK)\\
& \to H^{\Z^n}_q(E_{\vc};\bfK)\\
& \to H^{\Z^n}_q(E_{\vc},E_1;\bfK)
\end{align*}
The only map which is not explicit is the isomorphism $\cong$.  This only depends on an identification of $\Z^n/C$ with $\Z^{n-1}$ and the axioms of a generalized homology theory.

Here note for that for a generalized homology theory $\cH$, there is a canonical identification
\begin{equation}\label{circle}
\cH_q(S^1) = \cH_q(\pt) \oplus \cH_{q-1}(\pt)
\end{equation}
and hence a similar identification $\cH_q(T^n) = \bigoplus \binom{n}{i} \cH_{q-i}(\pt)$.  The identification \ref{circle} uses the fact that a circle has a point as a retract and the isomorphisms
$$
\cH_q(S^1,\pt) \xleftarrow{\cong} \cH_q(D^1,S^0)  \xrightarrow{\cong} \cH_{q-1}(S^0,\{-1\})  \xleftarrow{\cong} \cH_{q-1}(\{+1\}).
$$
\end{remark}

\begin{remark}
The group $\Z^n$ satisfies the property that every virtually cyclic subgroup is contained in a unique maximal virtual cyclic subgroup.  For such a group $G$, L\"uck-Weiermann \cite[Section 6]{LW} have shown $H^G_*(E_{\vc}G,E_{\fin}G;\bfK) \cong \bigoplus_{C \in \cM} H^C_*(E_{\vc}C,E_{\fin}C;\bfK)$ where $\cM$ is a set of representatives for the conjugacy classes of maximal virtually cyclic subgroups of $G$.   This also can be analyzed using the techniques of Davis-L\"uck \cite[Section 4]{DL03}.
\end{remark}






\section{Proof of Theorem \ref{infinite_dihedral}}

Recall the statement of Theorem \ref{infinite_dihedral}.

\begin{theorem5}
Let $\varphi : \g \to D_{\infty}$ be an epimorphism of groups.  Let $F = \ker \varphi$. Choose $\hat t \in \g$ so that $\varphi(\hat t) = ab$.    Let $\alpha$ be the automorphism of $RF$ given by conjugation by $\hat t$.  Then
$$
H^\g_q(E_{\varphi^*\cfin},E_{\varphi^*\fin}; \bfK) \cong NK_q(RF,\alpha).
$$
\end{theorem5}

\begin{proof}

Let $\varphi : \g \to D_{\infty}$ be an epimorphism.  The infinite dihedral group $D_{\infty} = \langle a, b ~|~ a^2 = 1 = b^2\rangle$ acts on $\R$ via $a(x) = -x$ and $b(1/2 + x) = 1/2 - x$ as well as on $S^\infty$ via $a(x) = -x = b(x)$.  Then one can choose models
\begin{align*}
E_{\varphi^*\fin}\g &= E_{\fin}D_\infty = \R \\
E_{\varphi^*{\cfin}}\g & = E_{\cfin}D_\infty = S^{\infty} * \R.
\end{align*}
This join model was pointed out by Ian Hambleton.  

Hence 
$$
H^\g_q(E_{\varphi^*\cfin}\g,E_{\varphi^*\fin}\g; \bfK) = H^\g_q(\R,S^{\infty} * \R; \bfK)
$$
To compute this relative homology we make a categorical diversion. 

For a subgroup $H$ of $G$ and a family $\cF$ of subgroups of $G$, let $\cF \cap H$ denote the family of subgroups of $H$ given by $\{F \in \cF : F \text{ is a subgroup of }H\}$.

\begin{lemma} \label{spectra} Let $\cF$ be a family of subgroups of $G$.  Let $\bfE$ be an \linebreak$\Or G\text{-spectrum}$, that is, a functor $\bfE : \Or G \to \spectra$.  
 There is an $\Or G$-spectrum $\bfE_{\cF}$ and 
a map of $\Or G\text{-}\spectra$  $\bfE_{\cF} \to \bfE$
 satisfying the following two properties:
\begin{enumerate}
\item For any subgroup $H$ of $G$, one can identify the change of spectra map
$$
H_q^G(G/H; \bfE_{\cF}) \to H_q^G(G/H; \bfE)
$$
with  the change of space map
$$
H^H_q(E_{\cF \cap H}H;j^*\bfE) \to H^H_q(\pt;j^*\bfE) = \pi_q\bfE(H/H).
$$
By $j^*\bfE$ we mean the composite 
$\Or H \xrightarrow{j} \Or G \xrightarrow{\bfE} \spectra$ where $j(H/K) = G/K$.

\item  For any family $\cG$ containing $\cF$, one can identify the change of spectra map 
$$
H^G_q(E_{\cG}G; \bfE_{\cF}) \to H^G_q(E_{\cG}G; \bfE)
$$
with the change of space map
$$
H^G_q(E_{\cF}G; \bfE) \to H^G_q(E_{\cG}G; \bfE).
$$
\end{enumerate}
\end{lemma}

We will prove a generalization of this lemma later, but for now note that
$$
\bfE_{\cF}(G/H) = \map_H(-,E_{\cF\cap H}H)_+   \wedge_{\Or H} \bfE(-).
$$

Let $\varphi: \g \to D_{\infty}$ be an epimorphism.  Let $\g = G_0 *_F G_1$ be the corresponding amalgamated product decomposition.  Then $\g_0 = \varphi^{-1}(\Z) = F \rtimes_{\alpha} \Z$ is an index 2 subgroup.  Conjugation by any element $g \in \g - \go$ leaves
$$
N_+K_q(RF,\alpha) \oplus N_-K_q(RF,\alpha^{-1}) \subset K_q(R\go)
$$ invariant and interchanges the two summands.  

Let $\bfN$ be the $\Or \g$-spectrum given as the cofiber of 
$$
\bfK_{\varphi^*\fin} \to \bfK.
$$
(The letter $\bfN$ is used to remind the reader of Nil.  A similar construction was in a preprint of Frank Quinn.)  For $H \in \varphi^*\fin$, note $E_{\varphi^*\fin \cap H}H = \pt$, so by Lemma \ref{spectra}(1),
$$
\pi_q\bfN(\g/H) = 0.
$$
By Lemma \ref{spectra}(1) and Lemma \ref{dlnil},
\begin{align*}
\pi_q\bfN(\g/\go) & = H^\g_q(E_{\all}\g, E_{\varphi^*\fin}\g; \bfK)\\
&= NK_q(RF,\alpha) \oplus  NK_q(RF,\alpha^{-1})\\
\end{align*}

Now finally we can prove Theorem \ref{infinite_dihedral}.
\begin{align*}
H^\g_q(E_{\varphi^*\cfin}\g,E_{\varphi^*\fin}\g; \bfK) & = H^\g_q(E_{\varphi^*\cfin}\g; \bfN)\\
& = H^\g_q(S^{\infty}*\R; \bfN)\\
& = H_q^\g(S^\infty; \bfN)\\
&= H_0(\Z_2; NK_q(RF,\alpha) \oplus NK_q(RF,\alpha^{-1}))\\
& = NK_q(RF,\alpha)
\end{align*}

We label the Equalities 1, 2, 3, 4, and 5.  Equality 1 follows from Lemma \ref{spectra}(2).  Equality 2 follows from the given model of the classifying space.  

For $x \in S^{\infty} * \R - \R$, the isotropy subgroup $\g_x \in \varphi^*\fin\g$, so $\pi_q\bfN(\g/\g_x) = 0$.  Thus the $E_2$-terms of the spectral sequences of the skeleta fitrations (see \cite[Theorem 4.7]{DL98})  converging to $H_q^\g(S^\infty; \bfN)$ and $H^\g_q(S^{\infty}*\R; \bfN)$ agree, so Equality 3 follows.  In fact, since the isotropy is constant, the $E_2$-term is
$$
E^2_{i,j} = H_i(\R P^{\infty}; NK_j(RF,\alpha) \oplus NK_j(RF,\alpha^{-1})),
$$
and computing with local coefficients, Equalities 4 and 5 follow.  

This completes the proof of Theorem \ref{infinite_dihedral}.

\end{proof}

\subsubsection{A categorical diversion}

We need to review some of the material in \cite{DL98} to state and prove the next lemma and deduce Lemma \ref{spectra}.  Let $\cC$ be a category.  A {\em $\cC$-space} is a functor $X : \cC \to \spaces$; a {\em $\cC$-spectrum} is a functor $\bfE: \cC \to \spectra$.    A map of $\cC$-spaces or $\cC$-spectra is a natural transformation.  A homotopy of maps of $\cC$-spaces is a map $X \times I \to Y$.  It is then clear what a homotopy equivalence of $\cC$-spaces is.  A {\em weak homotopy equivalence} $X \to Y$ of $\cC$-spaces is a map which induces a weak homotopy equivalence $X(c) \to Y(c)$ for all objects $c$ of $\cC$.  

A $\cC$-CW-complex is a  $\cC$-space $X$ together with a filtration 
$$
X^0 \subset X^1 \subset X^2 \subset \cdots \subset X
$$
satisfying certain properties; the precise definition is given in \cite{DL98}.  For example, if $Y$ is a $G$-CW-complex, then $\map_G(-,Y)$ is a $\Or G$-CW-complex.  A {\em $\cC$-CW approximation} is a weak homotopy equivalence $X' \to X$ where $X'$ is a $\cC$-CW-complex.  $\cC$-CW-approximations exist and are unique up to homotopy.

Let $X : \cC^{\op} \to \spaces$ be a $\cC^{\op}$-space and $\bfE : \cC \to \spectra$ be a $\cC$-spectrum.  One can form the balanced product 
$$
X_+ \wedge_{\cC} \bfE;
$$
this is a spectrum.  Let $X' \to X$ be a $\cC^{\op}$-CW-approximation.  One defines 
$$
H_q^{\cC}(X;\bfE) = \pi_q(X^\prime_+ \wedge_{\cC} \bfE).
$$
This generalized homology theory satisfies excision, is invariant under weak homotopy equivalence, has ``coefficients''
$$
H_q^{\cC}(\mor_{\cC}(-,c);\bfE) = \pi_q(\bfE(c))
$$ 
and satisfies 
$$
H_q^{\cC}(*;\bfE) = \pi_q(\hocolim_{\cC} \bfE),
$$
where $*$ denotes a $\cC$-space so that $*(c)$ is a point for all objects $c$. 

To explicate this last point, recall there is a $\cC^{\op}$-CW-approximation $E\cC \to *$, functorial in $\cC$.  Let $B\cC$ be the classifying space of a category $\cC$; it is the geometric realization of the simplicial set $\cN_\bullet \cC$ whose $p$-simplicies are sequences of composable morphisms 
$$
c_0 \to c_1 \to \cdots \to c_p.
$$
Fixing a object $c$ of $\cC$, define the {\em undercategory} $c \da \cC$.  An object in $c \da \cC$ is a morphism $\phi' : c \to c'$ in $\cC$.  A morphism $f$ from  $\phi' : c \to c'$ to $\phi'' : c \to c''$ is a morphism $f: c' \to c''$ satisfying $f \circ \phi' = \phi''$.  Then define the bar resolution model $E\cC(c) = B(c \da \cC)$.  Note $E\cC_+ \wedge_{\cC} \bfE$ is the usual definition of $\hocolim_{\cC} \bfE$.

Let $F : \cB \to \cC$ be a functor.  For a $\cC$-space $X$, define a $\cB$-space $F^*X(b) = X(F(b))$.  For a $\cB$-space $X$, define a $\cC$-space 
$$
F_*X(c) = \mor_{\cC}(F(-),c) \times_{\cB} X(-).
$$
There are similar definitions for spectra.  These constructions satisfy numerous adjoint properties.  If $X$ is a $\cB^{\op}$-space and $\bfE$ is a $\cC$-spectrum, there is a homeomorphism of spectra
$$
X_+ \wedge_{\cB} F^*\bfE \cong (F_*X)_+ \wedge_{\cC} \bfE,
$$
natural in $X$ and $\bfE$.  Similarly, if $X$ is a $\cB$-space and $Y$ is a $\cC$-space
$$
\map_{\cB}(X,F^*Y) \cong \map_{\cC}(F_*X, Y),
$$
natural in $X$ and $Y$.

Next we move on to assembly maps.  The functor $(b \da \cB) \to (F(b) \da \cC), \\\quad (\phi : b \to b') \mapsto (F(\phi) : F(b) \to F(b'))$ induces a map of $\cB$-spaces $E\cB \to F^*E\cC$, and hence, by the above adjoint property, a map of $\cC$-spaces
$$
F_*E\cB \to E\cC.
$$
We call this the {\em $F$-pre-assembly map}.  The composite 
$$
E\cB_+ \wedge_{\cB} F^*\bfE \cong F_*E\cB_+ \wedge_{\cC} \bfE \to E\cC_+ \wedge_{\cC} \bfE,
$$
as well as the induced map on homotopy groups
$$
H^{\cB}_q(*;F^*\bfE) \to H^{\cC}_q(*;\bfE),
$$
is called the {\em $(F,\bfE)$-assembly map}.

We need some more notation for the statement and proof of the following lemma.  Let $F : \cB \to \cC$ and $\bfE : \cC \to \spectra$ be functors.  For an object $c$ of $\cC$, define the overcategory  $F \da c$, whose objects are pairs $(b,\phi: F(b) \to c)$.  There is a commutative diagram of functors

$$
\begin{diagram}
\node{F \da c} \arrow{e,t}{F_c} \arrow{s,l}{P_c} \arrow{se,t}{\Delta_c} \node{\cC \da c} \arrow{s,r}{Q_c} \\
\node{\cB} \arrow{e,t}{F} \node{\cC} \arrow{e,t}{\bfE} \node{\spectra}
\end{diagram}
$$
where
\begin{align*}
P_c(b,\phi) &= b\\
F_c(b,\phi) & = \phi\\
Q_c(\phi: c' \to c) & = c'.
\end{align*}

\begin{lemma} \label{abstract}
Let $F : \cB \to \cC$ and $\bfE: \cC \to \spectra$ be functors.  There is a $\cC$-spectrum $\bfE_F : \cC \to \spectra$ with  $\bfE_F(c) = E(F\da c)_+ \wedge_{F\da c} \Delta_c^*\bfE$ and a map of $\cC$-spectra $\bfE_F \to \bfE$ satisfying the following properties:
\begin{enumerate}
\item For all objects $c$, the map $\pi_q\bfE_F(c) \to \pi_q\bfE(c)$ can be identified with the $(F_c,Q_c^*\bfE)$-assembly map
$$
H^{F\da c}_q(*; \Delta^*_c \bfE) \to H^{C\da c}_q(*; Q^*_c \bfE).
$$

\item For any $\cC$-space $X$, there is an isomorphism
$$
H^{\cC}_q(X; \bfE_F) \cong H^{\cB}_q(F^*X;F^*\bfE),
$$
natural in $X$ and $\bfE$.

\item  The change of spectrum map 
$$
H^{\cC}_q(*;\bfE_F) \to H^{\cC}_q(*;\bfE)
$$
can be identified with the $(F,\bfE)$-assembly map
$$
H^{\cB}_q(*;F^*\bfE) \to H^{\cC}_q(*;\bfE).
$$

\end{enumerate}
\end{lemma}

\begin{proof}[Proof of Lemma \ref{spectra} assuming Lemma \ref{abstract}]
Let $\cF$ be a family of subgroups of $G$.  Let $\Or(G,\cF)$ be the {\em restricted orbit category}; objects are $G$-sets $G/H$ with $H \in \cF$ and morphisms are $G$-maps.  Note that $\map_G(-,E_{\cF}G)$ is a model for $E\Or(G,\cF)$.

Lemma \ref{spectra}(1) follows immediately from applying Lemma \ref{abstract}(1) to the inclusion functor $F : \Or(G,\cF) \to \Or(G)$ and setting $\bfE_{\cF} = \bfE_F$, after noting the identification of categories 
\begin{align*}
\Or(H,\cF \cap H) & = F \da (G/H)\\
G/K & \mapsto (G/K \to G/H, \gamma K \mapsto \gamma H)
\end{align*}

For families $\cF \subset \cG$ of subgroups of $G$, let $F$ and $I$ be the inclusion functors
$$
\Or(G,\cF) \xrightarrow{F} \Or(G,\cG) \xrightarrow{I} \Or G.
$$
Lemma \ref{spectra}(2) follows from applying Lemma \ref{abstract}(3) to the functor $F$ and setting $\bfE_{\cF} = \bfE_{I \circ F}$.  We won't explicitly carry out the straightforward but tedious identifications of the maps in Lemma \ref{abstract}(3)  with the maps in Lemma \ref{spectra}(2), but will mention two identifications used:
\begin{align*}
\pi_q(\map_G(-,E_{\cG}G)_+ \wedge_{\Or(G,\cG)} I^*\bfE) & =
\pi_q(\map_G(-,E_{\cG}G)_+ \wedge_{\Or G } I^*\bfE)\\
(I^*{\bfE})_F & = I^*(\bfE_{I \circ F})
\end{align*}
\end{proof}

\begin{proof}[Proof of Lemma \ref{abstract}]
Define $\bfE_F(c) = E(F\da c)_+ \wedge_{F\da c} \Delta_c^*\bfE$.  For a morphism $f : c \to c'$ the map $\bfE_F(f) : \bfE_F(c) \to \bfE_F(c')$ is given by the $((F\da f) : F \da c \to F \da c', \Delta^*_{c'}\bfE)$-assembly map.  This defines $\bfE_F$.

The map $\bfE_F(c) \to \bfE(c)$ is given by the composite of the $(F_c, Q_c^*\bfE)$-assembly map
$$
E(F\da c)_+ \wedge_{F \da c} \Delta^*_c\bfE \to E(\cC \da c)_+ \wedge_{\cC  \da c} Q^*_c\bfE
$$
and the homotopy equivalence
$$
E(\cC \da c)_+ \wedge_{\cC  \da c} Q^*_c\bfE  \xrightarrow{\sim} \mor_{\cC \da c}(-,\id_c) \wedge_{\cC \da c} Q^*_c\bfE = Q^*_c(\bfE)(\id_c) = \bfE(c)
$$
which happens since $\id_c$ is a final object of $\cC \da c$.  This defines the map $\bfE_F \to \bfE$ and justifies (1).

We will next prove (2).  
There is a map of $(F \da c)$-spaces
\begin{align*}
E(F \da c) &\to P_c^*\mor_{\cC}(F(-),c)\\
x \in E(F\da c)(b,F(b) \to c) & \mapsto (F(b) \to c) \in P_c^*\mor_{\cC}(F(-),c)(b,F(b) \to c) \\
& \qquad= \mor_{\cC}(F(b),c)
\end{align*}
The adjoint of this map is a map
$$
P_{c*}E(F\da c) \to \mor_{\cC}(F(-),c),
$$
which, according to \cite[p. 91]{DL03}, is a weak homotopy equivalence of $\cB^{\op}$-spaces.

To prove (2), we may assume that $X$ is a $\cC^{\op}$-CW-complex, since both sides of (2) are invariant under weak homotopy equivalence.  According to Lemma 3.5 of \cite{DL03}, the domain of 
$$
X \times_{\cC} P_{c*}E(F\da c) \to X \times_{\cC} \mor_{\cC}(F(-),c) = F^*X
$$
has the homotopy type of a $\cB^{\op}$-CW-complex.  Hence the above map is a homotopy $\cB^{\op}$-CW-approximation.
Thus
$$
H^{\cB}_q(F^*X; F^*\bfE) = \pi_q((X \times_{\cC} P_{c*}E(F \da c))_+ \wedge_{\cB} F^*\bfE)
$$
Note
\begin{align*}
X_+ \wedge_{\cC} \bfE_F(c) & = X_+ \wedge_{\cC} (E(F\da c)_+ \wedge_{F \da c} P_{c}^*F^*\bfE)  \\
& = X_+ \wedge_{\cC} (P_{c*}E(F\da c)_+ \wedge_{\cB} F^*\bfE) \\
& = (X \times_{\cC} P_{c*}E(F\da c))_+ \wedge_{\cB} F^*\bfE 
\end{align*}
(2) follows.

We will construct maps of spectra
\begin{equation}
\begin{diagram}\label{triangle}
\node{E\cC_+ \wedge_\cC (\bfE(F\da c)_+ \wedge_{F\da c} \Delta_c^* \bfE)} \arrow{s} \arrow{e} \node{E\cC_+ \wedge_\cC \bfE}\\
\node{E\cB_+ \wedge_\cB F^*\bfE} \arrow{ne}
\end{diagram}
\end{equation}
 The horizontal map is defined given by the map of $\cC$-spectra $\bfE_F \to \bfE$ defined above.     The diagonal map is the $(F,\bfE)$-assembly map.  
The vertical homotopy equivalence was essentially defined in the proof of (2) with $X = E \cC$:  Recall 
$$
E\cC_+ \wedge_\cC (\bfE(F\da c)_+ \wedge_{F\da c} \Delta_c^* \bfE) = 
(E\cC \times_{\cC} P_{c*}E(F \da c))_+ \wedge_{\cB} F^*\bfE
$$
Recall also $E\cC \times_{\cC} P_{c*}E(F \da c) $ has the homotopy type of a $\cB^{\op}$-CW-complex and is contractible at each object, so there exists a $\cB^{\op}$-map to $E\cB$, unique up to homotopy, which is a homotopy equivalence.  Then smash with $F^*\bfE$.

Mike Mandell pointed out that the triangle (\ref{triangle}) only commutes up to homotopy.  He also indicated the proof of homotopy commutativity given below.

The proof uses simplicial methods.  In fact, the triangle (\ref{triangle}) is the geometric realization of a triangle of maps of simplicial spectra
\begin{equation}
\begin{diagram}\label{simp_triangle}
\node{\bfX_{\bullet}} \arrow{s,l}{h} \arrow{e,t}{f} \node{\bfZ_{\bullet}}\\
\node{\bfY_{\bullet}} \arrow{ne,r}{g}
\end{diagram}
\end{equation}
Let 
$$
\bfZ_p = \bigvee_{\sigma\in \cN_p \cC} \bfE(\initial \sigma)
$$
where $\initial (c_0 \to \dots \to c_p) = c_0$.  Similarly, let
$$
\bfY_p = \bigvee_{\sigma\in \cN_p \cB} \bfE(F(\initial \sigma))
$$
The upper left hand corner of (\ref{triangle}) is the geometric realization of a bisimplicial spectrum.  Let $\cN F_{p,q}$ be the bisimplicial set with elements
$$
\tau = (\tau_1, \tau_2) = (b_0 \to \cdots \to b_p, F(b_p) \to c_0 \to \dots \to c_q),
$$
sequences of composable morphisms in $\cB$ and $\cC$.  Let 
$$
\bfX_{p,q} = \bigvee_{(\tau_1,\tau_2)\in \cN F_{p,q}} \bfE(F(\initial \tau_1))
$$
Let $\bfX_\bullet$ be the diagonal simplicial set $X_p = X_{p,p}$.  A fundamental fact \cite[p. 94]{Q} is that the geometric realizations of $\bfX_{\bullet}$ and $\bfX_{\bullet,\bullet}$ are homeomorphic.

The geometric realizations of $\bfX_\bullet, \bfY_\bullet$, and $\bfZ_\bullet$ are homeomorphic to the corners of the triangle (\ref{triangle}).

The map $f$ in (\ref{simp_triangle}) is defined by mapping the 
$$
(\tau_1,\tau_2) = (b_0 \to \cdots \to b_p, F(b_p) \to c_0 \to \cdots \to c_p)
$$
summand of $\bfX_p$ to the ($c_0 \to \cdots \to c_p$)-summand of $\bfZ_p$ using the map
$$
\bfE(F(b_0)) \to \bfE(c_0)
$$
induced by the composite 
$$
F(b_0) \to \cdots \to F(b_p) \to c_0.
$$
The map $h$ in (\ref{simp_triangle}) (using the obvious choice of a map of $\cB$-spaces $E\cC \times_{\cC} P_{c*}E(F\da c) \to E\cB$) maps the $(\tau_1,\tau_2)$-summand of $\bfX_p$ to the $\tau_1$-summand of $\bfY_p$ using the identity map
$$
\bfE(F(b_0)) \to \bfE(F(b_0)).
$$
The map $g$ in (\ref{simp_triangle}) maps the $(b_0 \to \cdots \to b_p)$-summand of $\bfY_p$ to the $(F(b_0) \to \cdots \to F(b_p))$-summand of $\bfZ_p$ using the identity map
$$
\bfE(F(b_0)) \to \bfE(F(b_0)).
$$

A simplicial homotopy between $f$ and $g \circ h$ is a sequence of maps of spectra
$$
H_i : \bfX_p \to \bfZ_{p+1} \qquad (i = 0, \dots , p)
$$
satisfying certain identities (see e.g.~\cite[Definition 5.1]{May}), including $\partial_0 H_0 = f$ and $\partial_{p+1}H_p = g \circ h$.  In our case, $H_i$ maps the 
$$
(b_0 \to \cdots \to b_p, F(b_p) \to c_0 \to \cdots \to c_p)
$$
summand of $\bfX_p$ to the 
$$
F(b_0) \to \cdots \to F(b_i) \to c_i \to \cdots \to c_p
$$
summand of $\bfZ_p$ using the identity map
$$
\bfE(F(b_0)) \to \bfE(F(b_0))
$$
for $i \not = 0$ and the composite induced map 
$$
\bfE(F(b_0)) \to \bfE(c_0)
$$
for $i = 0$.  This provides the desired homotopy.

\

\end{proof}

\section{Consequences of Quinn's Theorems} \label{Quinn}

In this section we prove the corollaries that we mentioned in the first section, namely Corollaries \ref{laurent}, \ref{kregular}, and \ref{dv}.  The first two corollaries depend on the recent proof of the isomorphism conjecture in $K$-theory for $G= \Z^n$:

\begin{theorem}[Quinn \cite{Q1}, 2.4.1 and 3.3.1] \label{Quinn1} 
$H_*^{\Z^n}(E_{\all},E_{\vc}; \bfK) = 0$.
\end{theorem}

\begin{proof}[Proof of Corollary \ref{laurent}]

Note that for the group $G = \Z^n$,
$$
E_1 = E_{\fin} = E\Z^n;
$$
a useful model is $\R^n$ with $\Z^n$ acting by translations.  We show below that exact sequence of the pair
$$
0 \to H^{\Z^n}_q(E_{1}; \bfK) \to H^{\Z^n}_q(E_{\all}; \bfK) \to H^{\Z^n}_q(E_{\all},E_{1}; \bfK) \to 0
$$
is short exact and split. Assuming this for now, we see
\begin{align*}
K_qR[\Z^n] &= H_q^{\Z^n}(E_{\all}; \bfK) \\
& \cong H_q^{\Z^n}(E_1; \bfK) \oplus H^{\Z^n}_q(E_{\all},E_{1}; \bfK) \\
& = H_q(B\Z^n; \bfK(R)) \oplus H^{\Z^n}_q(E_{\vc},E_{1}; \bfK) \\
& \cong \bigoplus_i \binom{n}{i} K_{q-i}R  \oplus \bigoplus_i 2\binom{n-1}{i} NK_{q-i}R  \quad \text{ using Theorem \ref{Zn}}
\end{align*}

Thus we just need to show that the horizontal maps
$$
\begin{diagram}
  \node{H^{\Z^n}_q(E_{1}; \bfK)}  \arrow{s,r}{\cong} \arrow{e} \node{H^{\Z^n}_q(E_{\all}; \bfK)}  \arrow{s,r}{\cong} \\
  \node{\bigoplus_i \binom{n}{i} K_{q-i}R} \arrow{e} \node{K_qR[\Z^n]}
\end{diagram}
$$
are split injective.  We momentarily write $\bfK_R$ instead of $\bfK$\linebreak (so $\pi_q \bfK_R(G/H) = K_qRH$).  We prove the splitting exists by working inductively on $n$, with the inductive step given by the composite map
\begin{align*}
H_q^{\Z^n}(E_{\all}; \bfK_{R[\Z^{n}]}) & = K_qR[\Z^n] \\
& = H_q^{\Z}(E_{\all}; \bfK_{R[\Z^{n-1}]})\\
& \to H_q^{\Z}(E_{1}; \bfK_{R[\Z^{n-1}]}) = K_qR[\Z^{n-1}] \oplus K_{q-1}R[\Z^{n-1}] \\
& = H_q^{\Z^{n-1}}(E_{\all}; \bfK_{R[\Z^{n-1}]}) \oplus H_{q-1}^{\Z^{n-1}}(E_{\all}; \bfK_{R[\Z^{n-1}]}),
\end{align*}
where the map ``$\to$'' is given applying Lemma \ref{dlnil}(1).
\end{proof}

There is another decomposition of $K_qR[\Z^n]$ given by the fundamental theorem of $K$-theory.  We will review this to compare and contrast with our decomposition and to set the stage for the next corollary.
The fundamental theorem of algebraic $K$-theory states that 
$$
K_qR[t,t^{-1}] \cong K_q R \oplus NK_qR \oplus NK_qR \oplus K_{q-1}R.
$$
More precisely, for any functor $K : \rings \to \ab$, Bass \cite[Chapter XII, Section 7]{B} defines two functors 
\begin{align*}
NK(R) &= \ker(K(R[t]) \to K(R))\\
LK(R) & = \text{cok}(K(R[t]) \oplus K(R[t^{-1}]) \to K(R[t,t^{-1}])) 
\end{align*}
Bass calls $K$  a {\em contracted functor} if the four-term sequence
$$
0 \to K(R) \to K(R[t]) \oplus K(R[t^{-1}]) \to K(R[t,t^{-1}]) \to LK(R) \to 0
$$
is exact and the surjection with target $LK(R)$ is split, naturally in $R$.   He noted that $NLK(R) \cong LNK(R)$ and that if $K$ is contracted, then so are $NK$ and $LN$.  The more precise version of the fundamental theorem of $K$-theory is that $K_q$ is a contracted functor ($q \in \Z$) and there is a natural identification $K_{q-1} = LK_q$, which in fact is taken as the definition of $K_q$ for $q$ negative.

This allows a very cute formulation of the calculation of the $K$-theory of (Laurent) polynomial rings.

\begin{align}
\label{ftiso} K_q(R[t_1, t_1^{-1}, \dots t_n, t_n^{-1}]) & \cong (I + 2N + L)^n K_q(R) \\
\label{ftpoly} K_q(R[t_1, \dots t_n]) & \cong (I + N)^n K_q(R)
\end{align}

\begin{proof}[Proof of Corollary \ref{kregular}]
First one needs to identify the change of space map
$$
H^{\Z^n}_q(E_1;\bfK) \to H^{\Z^n}_q(E_{\all};\bfK)
$$
with the split injection from the fundamental theorem of $K$-theory
$$
\bigoplus_i \binom{n}{i} K_{q-i}R = (I + L)^n K_qR \to K_qR[\Z^n].
$$
This follows from the case where $n=1$ which can be done directly (using that the map $K_{q-1}R \to K_qR[t,t^{-1}]$) is given by the product with $t \in K_1\Z[t,t^{-1}]$ or by consulting \cite[section 4 and 8]{HP}. 

Thus the isomorphism 
$$
\bigoplus_{n \geq i \geq 0} \binom{n}{i} K_{q-i}R~ \oplus  ~\bigoplus_{C \in \cM}  ~ \bigoplus_{n-1 \geq i \geq 0} 2\binom{n-1}{i} NK_{q-i}R \cong (I+2N+L)^nK_qR
$$
given by comparing the isomorphisms from Corollary \ref{laurent} and equation (\ref{ftiso}) restricts to the ``identity'' on the subgroups $\bigoplus_{n \geq i \geq 0} \binom{n}{i} K_{q-i}R \to (I+L)^nK_qR$.  This induces an isomorphism on the quotient
$$
 ~\bigoplus_{C \in \cM}  ~ \bigoplus_{n-1 \geq i \geq 0} 2\binom{n-1}{i} NK_{q-i}R \cong ((I+2N +L)^n - (I+L)^n)K_qR
$$
Thus if the left hand side vanishes $N^iK_qR = 0$ for $n \geq i \geq 1$.  Hence by equation (\ref{ftpoly}), $K_qR[t_1,\ldots t_n] = K_qR$ as desired.
\end{proof}

\begin{remark} \label{explicit}
Since $K_qR[\Z^n]$ is, in general, an infinitely generated abelian group, it is worth being more explicit about the isomorphism in Corollary \ref{laurent}.  To this end, let $M(n)$ be the set of  degree $n$ monomials in the {\em non-commuting} variables $I$ and $L$.  Let $M(n,i) \subset M(n)$ be those monomials with exactly $i$ $L$'s.  

For each maximal infinite cyclic subgroup $C$ of $\Z^n$, choose a basis $b_1, b_2, \ldots, b_n$ of $\Z^n$ with $b_1 \in C$, i.e. choose an automorphism $\beta_C : \Z^n \to \Z^n$ with $\beta_C(\Z \times \{(0,\ldots, 0)\}) = C$.  We define an internal direct sum decomposition
$$
K_qR[\Z^n] = \bigoplus_{f \in M(n)} fK_qR \oplus \bigoplus_{C \in \cM} \bigoplus_{f \in M(n-1)} \beta_{C*}(fN_+K_qR \oplus fN_-K_qR).
$$
We indicate the summands by an example:  
$$LILK_qR \subset K_qR[t_1,t_1^{-1},t_2,t_2^{-1},t_3,t_3^{-1}]$$ is given by
\begin{align*}
(\cdot t_3) \text{inc}_* (\cdot t_1)K_{q-2}R & \subset
(\cdot t_3) \text{inc}_* K_{q-1}R[t_1,t_1^{-1}]\\
& \subset (\cdot t_3) K_qR[t_1,t_1^{-1},t_2,t_2^{-1}]\\
& \subset K_qR[t_1,t_1^{-1},t_2,t_2^{-1},t_3,t_3^{-1}]
\end{align*}
The proof of Corollary \ref{laurent} demonstrates the internal direct sum decomposition.    This direct sum decomposition may be easier to parse if we write it as 
$$
K_qR[\Z^n] = (I+L)^n K_qR \oplus \bigoplus_{C \in \cM}\beta_{C*}((I+L)^{n-1}(N_+ + N_-) K_qR)
$$
The $\cM$-summands depend only on the choice of a splitting of the injection $C \to \Z^n$.  If $f \in M(n,i)$ there is a canonical isomorphism
$$
fK_qR \cong K_{q-i}R
$$
and for $f \in M(n-1,i)$,
$$
fN_{\pm}K_qR \cong N_{\pm}K_{q-i}R.
$$
\end{remark}

\begin{remark}
Let $\cM_{\not = 0}$ be the set of maximal cyclic subgroups of $\Z^n$ whose generators have all nonzero coordinates.  Let $\cM_+$ be the set of maximal cyclic subgroups of $\Z^n$ with a generator having  all positive coordinates.  

The proof of Corollary \ref{kregular} shows that
$$
(N_+ + N_-)^n K_qR = \bigoplus_{C \in \cM_{\not = 0}}\beta_{C*} (I+L)^{n-1}(N_+ + N_-) K_qR.
$$
This implies 
$$
2^nN^nK_qR \cong  \bigoplus_{C \in \cM_{\not = 0}} \bigoplus_i 2\binom{n-1}{i}NK_{q-i}R.
$$
\end{remark}

\begin{corollary}  \label{countable}
Suppose $NK_qR, NK_{q-1}R, \dots , NK_{q-n+1}R$ are all countable torsion groups.  Then 
$$
N^nK_qR \cong \bigoplus_{C \in \cM_+} \bigoplus_i \binom{n-1}{i} NK_{q-i}R.
$$
\end{corollary}

\begin{proof}
There is a $2^{n-1}$-to-$1$ map $\cM_{\not = 0} \to \cM_+$ given by choosing a generator $(x_1, \dots, x_n)$ for $C$ and sending
$$
C = \langle (x_1, \dots, x_n) \rangle \in \cM_{\not = 0} \mapsto  \langle (|x_1|, \dots, |x_n|) \rangle \in \cM_{+}
$$ 
Thus
$$
2^nN^nK_qR \cong 2^n \bigoplus_{C \in \cM_+} \bigoplus_i \binom{n-1}{i}NK_{q-i}R.
$$
If $A$ and $B$ are countable torsion abelian groups, then 
$$
A \oplus A \cong B \oplus B \Rightarrow A \cong B
$$
by Ulm's Theorem (see \cite[Exercise 32]{K}).  The corollary follows.
\end{proof}

\begin{remark} 
If $R$ is a countable ring so that $R \otimes \Q$ is regular noetherian, then $NK_*R$ is countable torsion.  The fact that they are countable follows from the fact that the homology of $BGL(R[\Z^n])$ is countable.  The fact that they are torsion comes from the fact that $BGL(R[\Z^n]) \to BGL((R\otimes \Q)[\Z^n])$ is a rational equivalence, and that $NK_*(R \otimes \Q) = 0$.  This remark is due to Chuck Weibel.

A key example is $R = \Z G$ for $G$ a finite group.  
\end{remark}

\begin{remark}
Tom Farrell told the author some years ago that he and Lowell Jones knew that the isomorphism conjecture for $\Z^n$ has striking consequences for $N^iK$ and presumably anticipated Corollary \ref{laurent}.  Nonetheless, now that the isomorphism conjecture has been proved for $\Z^n$, it is worthwhile to document its consequences.  

 Corti\~nas, Haesemeyer, and  Weibel \cite{CHW} use radically different techniques to analyze $N^iKR$.  Their results are both  stronger, since they have a more complete computation using Hochschild homology and weaker, because their results apply to a restricted class of rings.  It would be interesting to compare the two approaches.
\end{remark}

In \cite{Q2}, Frank Quinn uses controlled topology to  prove the following theorem.  Holger Reich \cite{R} gave a short proof of the theorem using a theorem from \cite{BLR}.

\begin{theorem}[Quinn] \label{Quinn2} Let $\g \to D_{\infty}$ be an epimorphism of groups.  Then
$$
H^\g_*(E_{\all},E_{\phi^*\cfin}; \bfK) = 0.
$$
\end{theorem}

\begin{proof}[Proof of Corollary \ref{dv}]
\begin{align*}
NK_q(RF;\widehat{RG_0},\widehat{RG_1}) & = H^\g(E_{\all},E_{\varphi^*\fin};\bfK) \\
&= H^\g(E_{\varphi^*\cfin},E_{\varphi^*\fin};\bfK) \\
& = NK_q(RF,\alpha)
\end{align*}
The first equality holds by Lemma \ref{dlnil}(2), the second by Quinn's theorem, and the third by Theorem \ref{infinite_dihedral}.
\end{proof}

\begin{remark}
Several years ago the author and Bogdan Vajiac outlined a unpublished proof of Corollary \ref{dv}  for $q \leq 1$, using controlled topology.  The original motivation for this note was to  check that Quinn's Theorem \ref{Quinn2} is consistent with the work of Davis-Vajiac.  Previous partial results are contained in the papers \cite{FJ95}, \cite{JP}, \cite{CP}, \cite{LO}.  A proof of Corollary \ref{dv} avoiding controlled topology is given in \cite{DKR}.
\end{remark}

 This paper benefited from conversations with Tom Farrell, Christian Haesemeyer, Mike Mandell, Frank Quinn, and Chuck Weibel.

\end{document}